\documentclass{amsart}

\usepackage{amsfonts, amsthm, amscd,color,amsmath,amssymb,graphics,graphicx}
\usepackage{epic,eepic}
\DeclareMathOperator{\Lim}{Lim} \DeclareMathOperator{\diam}{diam}

\DeclareMathOperator{\inte}{int}

\newtheorem{theorem}{Theorem}[section]

\newtheorem{corollary}[theorem]{Corollary}
\newtheorem{proposition}[theorem]{Proposition}
\newtheorem{observation}[theorem]{Observation}

\theoremstyle{definition}

\newtheorem{example}[theorem]{Example}
\newtheorem{examples}[theorem]{Examples}

\theoremstyle{remark}
\newtheorem{remark}[theorem]{Remark}

\newtheorem{claim}{Claim}[theorem]

\numberwithin{equation}{section}

\begin{document}
\title{More absorbers in hyperspaces}

\author[P. Krupski]{Pawe\l\ Krupski}
\email{Pawel.Krupski@math.uni.wroc.pl}
\address{Mathematical Institute, University of Wroc\l aw, pl.
Grunwaldzki 2/4, 50--384 Wroc\l aw, Poland}
\author[A. Samulewicz]{Alicja Samulewicz}
\email{Alicja.Samulewicz@polsl.pl}
\address{Institute of Mathematics, Faculty of Applied Mathematics, Silesian University of Technology, ul. Kaszubska 23, 44-101 Gliwice, Poland}
\subjclass[2010]{Primary 57N20; Secondary 54H05, 54F45}
\keywords{absorber, aposyndetic, Borel set, coanalytic set, colocally connected, continuum, C-space, Hilbert cube, Hilbert cube manifold, hyperspace, infinite-dimensional space, locally connected, manifold, weakly infinite-dimensional space}

\begin{abstract} The family  of all subcontinua that  separate  a compact connected $n$-manifold $X$ (with or without boundary), $n\ge 3$,  is an $F_\sigma$-absorber in the hyperspace $C(X)$ of nonempty subcontinua of $X$.  If $D_2(F_\sigma)$ is the  small Borel class of spaces which are differences of two  $\sigma$-compact sets, then the family  of all $(n-1)$-dimensional continua that separate   $X$ is a $D_2(F_\sigma)$-absorber in $C(X)$. The families of nondegenerate colocally connected or aposyndetic continua in $I^n$ and of at least two-dimensional or decomposable Kelley continua   are $F_{\sigma\delta}$-absorbers in the hyperspace $C(I^n)$  for $n\ge 3$.  The hyperspaces  of all weakly infinite-dimensional continua and of $C$-continua of  dimensions at least 2 in a compact connected Hilbert cube manifold $X$ are $\Pi ^1_1$-absorbers in  $C(X)$.  The family of all hereditarily infinite-dimensional compacta in the Hilbert cube $I^\omega$ is  $\Pi ^1_1$-complete in $2^{I^\omega}$.
\end{abstract}

\maketitle

\section{Introduction}
The theory of absorbing sets  was well developed in the eighties and nineties of the last century (see~\cite{B} and~\cite{M}). Since any two absorbers in a Hilbert cube $\mathcal Q$ of a given Borel or projective class are homeomorphic via arbitrarily small ambient homeomorphisms of $\mathcal Q$, it provides a powerful technique of characterizing  some subspaces of  the hyperspaces  $2^X$ of all closed nonempty subsets of a nondegenerate Peano continuum $X$ or $C(X)$ of all subcontinua of  $X$ (if $X$ contains no free arcs). Nevertheless, the list of natural examples of such spaces is not too long and they usually are considered in $X$ being a Euclidean or the Hilbert cube.

Let $I=[0,1]$ be the closed unit interval with the Euclidean metric. The following subspaces of a respective hyperspace are absorbers of Borel class $F_\sigma$ (otherwise known as cap-sets), so they are homeomorphic to the pseudo-boundary $B(I^\omega)=\{(x_i)\in I^\omega: \,\exists\,i\, (x_i\in\{0,1\})\,\}$, a standard $F_\sigma$-absorber in the Hilbert cube $I^\omega$:
\begin{examples}\label{ex1}\hfill
\begin{enumerate}
\item
The subspace $\mathcal D_n(X)$ of $2^X$ consisting of all compacta of covering dimensions $\ge n\ge 1$, where $X$ is a locally connected continuum each of whose open non-empty subset has  dimension $\ge n$~\cite{C3} (for $X=I^\omega$, see~\cite{DMM}),
\item
The family of all decomposable continua in $I^n$, $n\ge 3$~\cite{Sam2},
\item
The family of all compact subsets  (subcontinua) with nonempty interiors in a locally connected nondegenerate continuum (containing no free arcs)~\cite{CM},
\item
The family of all compact subsets that block all subcontinua of  a locally connected nondegenerate continuum which is not separated by any  finite subset~\cite{IK}.
\end{enumerate}
\end{examples}

\begin{examples}\label{ex2}\hfill
Known $F_{\sigma\delta}$-absorbers  include two standard ones $(B(I^\omega))^\omega$ (in $(I^\omega)^\omega$) and  $\widehat{c_0}=\{(x_i)\in I^\omega: \lim_i x_i =0\}$ (in $I^\omega$)~\cite{DMM} and
\begin{enumerate}
\item
The subspace of $2^{I^\omega}$  of all infinite-dimensional compacta~\cite{DMM},
\item
The subspace of $C(I^n)$ of all locally connected subcontinua of  $I^n$, $n\ge 3$~\cite{GM},
\item
The subspaces of  $C(I^2)$ of all arcs~\cite{C2} and of all absolute retracts ~\cite{CDGM}.
\end{enumerate}
\end{examples}

\begin{examples}\label{ex4}
If $D_2(F_\sigma)$ is the class of all subsets of the Hilbert cube that are differences of two $F_\sigma$-sets, then
a standard  $D_2(F_\sigma)$-absorber is the subset $B(I^\omega)\times s$  of $I^\omega\times I^\omega$ ($s=I^\omega\setminus B(I^\omega)$); its  incarnation  in $2^{I^\omega}$ is the family $\mathcal D_n(I^\omega)\setminus \mathcal D_{n+1}(I^\omega)$, $n\ge 1$~\cite{DMM}.
\end{examples}

\begin{example}
The subspace of $2^{\mathbb R^n}$, $n\ge 3$, of all compact ANR's in $\mathbb R^n$ is a $G_{\delta\sigma\delta}$-absorber~\cite{DR}.
\end{example}

Let $\Pi ^1_1$ and $\Sigma_1^1$ denote the classes of coanalytic and analytic sets, respectively.
Concerning $\Pi ^1_1$-absorbers (coanalytic absorbers), the Hu\-re\-wicz set $\mathcal H$ of all countable closed subset of $I$ can be treated as a standard $\Pi ^1_1$-absorber in the hyperspace $2^I$~\cite{C1}.  Other known examples of coanalytic absorbers we wish to recall here are:
\begin{examples}\label{ex3}\hfill
\begin{enumerate}
\item
The subspace of $C(I^n)$, $3\le n\le \omega$, consisting of all hereditarily decomposable  continua~\cite{Sam1},
\item
The subspaces $\mathcal{SCD}_k$ of $2^{I^\omega}$  of all strongly countable-dimen\-sional compacta of dimensions $\ge k$ and  $\mathcal{SCD}_{k+1}\cap C(I^\omega)$ of $C(I^\omega)$ of continua of dimensions $> k $, for any $k\in \mathbb N$~\cite{KS},
\item
The families of Wilder continua, of continuum-wise Wilder continua and of hereditarily arcwise connected nondegenerate continua in $I^n$, $n\ge 3$~\cite{KK}.
\end{enumerate}
\end{examples}

The main purpose of this paper is to add to the list  some new examples of absorbers in the hyperspaces of cubes, $n$-manifolds or the Hilbert cube manifolds.

 In Section~\ref{apo} we deal with three important in continuum theory classes $Col$ of colocally connected, $Apo$ of aposyndetic and $\mathcal K$ of Kelley continua.  Class $Apo$  properly contains $Col$  and it is properly contained in the class of, so called, continua with property C~\cite{Wilder} which we call Wilder continua~\cite{KK}. The class $LC$ of locally connected  continua ($=$ Peano continua) is properly contained in both $Apo$ and $\mathcal K$.  We evaluate the Borel class of  families $Apo(Y)$ and $Col(Y)$ of nondegenerate aposyndetic and colocally connected, resp., subcontinua of a compact space $Y$ as $F_{\sigma\delta}$ in hyperspace $C(Y)$ and show that the families are  $F_{\sigma\delta}$-universal if $Y$ contains a 2-cell. An analogous result is known for the family  $\mathcal K(Y)$ of Kelley subcontinua of a Peano continuum $Y$~\cite[Theorem 3.3]{Kr1}. We prove that if $Y$ is a Peano continuum each of whose open subset contains an $n$-cell ($n\ge 2$), then $Apo(Y)$,   $\mathcal K(Y)$ and $LC(Y)$ are strongly $F_{\sigma\delta}$-universal in $C(Y)$ and the families restricted to at least 2-dimensional continua are  $F_{\sigma\delta}$-absorbers in $C(Y)$.      Moreover, $Col(I^n)$ and $Apo(I^n)$ are shown to be $F_{\sigma\delta}$-absorbers in $C(I^n)$, $n\ge 3$. The same proof  shows that also the family of all  decomposable Kelley continua  in $I^n$ are $F_{\sigma\delta}$-absorbers in $C(I^n)$ for $n\ge 3$.

In Section~\ref{sep} we show that if $X$ is a locally connected continuum such that each open non-empty subset of $X$ contains a copy of $(0,1)^n$, $3\le n<\infty$, as an open subset and no subset of dimension $\le 1$ separates $X$,  then the families $\mathcal S(X)$ of all closed  separators of $X$ and  $\mathcal S(X)\cap C(X)$ of all continua that separate $X$ are $F_\sigma$-absorbers in $2^{X}$ and $C(X)$, respectively. Nowhere dense closed separators  and nowhere dense continua which separate $X$ form  $D_2(F_\sigma)$-absorbers in $2^X$ and $C(X)$, respectively.

One of the central notions in the theory of infinite dimension is that of a weakly infinite-dimensional space introduced by P. S. Alexandroff in 1948. In Section~\ref{WID} we find  two new coanalytic absorbers in $2^X$  ($C(X)$) for each  locally connected continuum $X$ such that each open non-empty subset of $X$ contains a Hilbert cube:  the family  of weakly infinite-dimensional compacta (continua) in $X$ of dimensions $\ge n\ge 1$ ($\ge n\ge 2$) and  the family  of compacta (continua) which are $C$-spaces of dimensions $\ge n\ge 1$ ($\ge n\ge 2$).

We observe that the collection   of hereditarily infinite-dimen\-sional compacta in $I^\omega$ is a $\Pi^1_1$-complete subset of  $2^{I^\omega}$ contained in a $\sigma Z$-set. We do not know however if it is a coanalytic absorber in $2^{I^\omega}$.

\section{Preliminaries}
 All spaces in the paper are assumed to be metric separable. The hyperspace $2^X$ of nonempty compact subsets of $X$ is endowed with the Hausdorff metric $dist$ and the hyperspace $C(X)$ of continua in $X$ is considered as a subspace of $2^X$.

If $X$ is a locally connected nondegenerate continuum, then $2^X$ is homeomorphic to
$I^\omega$ and if, additionally,   $X$ contains no free arcs, then also $C(X)$  is a Hilbert cube~\cite{CS}.

A closed subset $C$ \emph{separates} a space $X$ between two disjoint subsets $A$ and $B$ if there are open disjoint subsets $U$ and $V$ of $X$ such that $A\subset U$, $B\subset V$ and $X\setminus C= U\cup V$. Such set $C$ will be called a \emph{closed separator} in $X$.

A subset $C$ \emph{cuts} $X$ between two disjoint subsets $A$ and $B$ ($C$ is a \emph{cut} between $A$ and $B$) if it is disjoint from  $A\cup B$ and any continuum  $D\subset X$ that meets both $A$ and $B$ intersects $C$.

If a compact space $X$ is locally connected, then a closed subset $C\subset X$ separates $X$ between  $A$ and $B$ if and only is $C$ cuts $X$ between  $A$ and $B$~\cite[Theorem 1, p. 238]{Ku}.

Suppose $X$ is a closed subset of a space $Y$. It is known that for each closed separator $C$ in $X$ between disjoint closed subsets $A$ and $B$ of $X$, there is a closed separator $C'$ in $Y$ between $A$ and $B$ such that $C'\cap X = C$~\cite[Lemma 1.2.9, Remark 1.2.10]{E}.

Combining these two classic facts, we get an easy observation.
\begin{observation}\label{o1}
Let $X$ be a closed subset of a compact locally connected space $Y$. A closed set $C\subset X$ separates $X$ between closed subsets $A$ and $B$ of $X$ if and only if
$C'$ cuts $Y$ between $A$ and $B$.
\end{observation}

Recall that, given a class $\mathcal M$ of spaces,  a subset  $A$ of a complete space $Z$ is {\it $\mathcal M$-complete}
if $A\in \mathcal M$ and $A$ is $\mathcal M$-\emph{hard}, i.e., for any  subset $C\in\mathcal M$ of a complete $0$-dimensional space $Y$ there is a continuous mapping $\xi: Y\to Z$ (called
a {\it reduction of $C$ to $A$}) such that $\xi^{-1}(A)=C$~\cite{Ke}.

A closed subset $B$ of a Hilbert cube $(\mathcal Q,d)$ is a $Z$-{\it set} in $\mathcal Q$ if
 for any $\epsilon>0$ there exists a continuous mapping $f:\mathcal Q\to \mathcal Q$
such that  $f(\mathcal Q)\cap B= \emptyset$
and  $\widetilde{d}(f,\operatorname{id}_\mathcal Q)=\sup\{d(f(x),x):x\in \mathcal Q\}<\epsilon$. A countable union of $Z$-sets in $\mathcal Q$ is called a $\sigma Z$-{\it set} in $\mathcal Q$.

Let $\mathcal M$ be a class of spaces which is topological (i.e., if $M\in \mathcal M $ then each homeomorphic image of $M$ belongs to $\mathcal M$) and closed hereditary (i.e., each closed subset of $M\in \mathcal M$ is in $\mathcal M$).
Following~\cite{DMM}, we call  a subset $A$  of a Hilbert cube $\mathcal Q$  \emph{$\mathcal M$-universal} if for each
$M\subset I^\omega$ from the class $\mathcal M$ there is an embedding $f:I^\omega \to\mathcal Q$ (a reduction of $M$ to $A$) such that $f^{-1}(A)=M$;
$A$ is said to be
\emph{strongly $\mathcal M$-universal}
if for each
$M\subset I^\omega$ from the class $\mathcal M$ and
each  compact set $K\subset I^\omega$, any embedding $f:I^\omega \to\mathcal Q$ such that
$f(K)$ is a $Z$-set in $\mathcal Q$
can be approximated arbitrarily closely
 by an embedding $g:I^\omega\to \mathcal Q$ such that
   $g(I^\omega)$ is a $Z$-set in $\mathcal Q$,
$g|K=f|K$ and $g^{-1}(A)\setminus K=M\setminus K$.

Observe that a strongly $\mathcal M$-universal set is $\mathcal M$-universal and an $\mathcal M$-universal set is $\mathcal M$-hard. Often in practice, in order to show that a set $A\subset \mathcal Q$ is $\mathcal M$-universal (or $\mathcal M$-hard), we choose an already known  $\mathcal M$-universal ($\mathcal M$-hard) set $B\subset I^\omega$ and construct a continuous embedding $\xi: I^\omega\to \mathcal Q$ such that $\xi^{-1}(A)=B$.

A subset $A$ of a Hilbert cube $\mathcal Q$ is called an  $\mathcal M $-\emph{absorber} in $\mathcal Q$ provided that:
\begin{enumerate}
\item $A\in \mathcal M$;
\item $A$ is contained in a $\sigma Z$-set in $\mathcal Q$;
\item $A$ is strongly $\mathcal M$-universal.
\end{enumerate}

\section{Proving strong $\mathcal M$-universality in hyperspaces}\label{su}

We are now going to sketch two techniques   for proving strong $\mathcal M$-universality. The first one, developed in~\cite{GM}, \cite[Lemma 3.2]{Sam1}, applies to  subsets $\mathcal A$ of $2^{I^n}$ or $C(I^n)$, $n\in \mathbb N\cup\{\omega\}$ (for simplicity we  will consider the case $n\in \mathbb N$). The second, presented in~\cite{C3}, concerns subsets $\mathcal A$ of the Hilbert cube $\mathcal Q=2^X$ or $\mathcal Q=C(X)$ in a much more general case  of a locally connected nondegenerate continuum $X$ (without free arcs) satisfying certain local properties. For some classes of continua located in such $X$  the second technique may fail while the first one still works if $X$ is a cube.

\subsection{Approach I}\label{ap1}
Suppose  that $K$ is a compact subset of $I^\omega$ and  $f:I^\omega \to C(I^n) (2^{I^n})$ is an embedding such that $f(K)$ is a $Z$-set and $\epsilon>0$. We have to find a $Z$-embedding $g:I^\omega \to C(I^n) (2^{I^n})$ which agrees with $f$ on $K$, is $\epsilon$-close to $f$ and satisfies
\begin{equation}\label{eq:uni}
g^{-1}(\mathcal A )\setminus K= M\setminus K.
\end{equation}

For a construction of $g$ we  need an auxiliary map
$\theta:I^\omega \to C([-1,1]^n)$
sending $q=(q_i)$ to
\begin{multline*}
\theta(q)= \\
 \left(([-1,0]\times \{0\})\cup S((-\frac12,0);\frac12)\cup \bigcup_{i=1}^\infty S(a_i;r_i(q))\right)\times \{(\underbrace{0,\dots, 0}_{n-2})\},
 \end{multline*}
where $S(x;r)$ denotes the circle in the plane centered at $x$ with radius $r$, $a_i=(-1+2^{-i},0)\in \mathbb R^2$ and $r_i(q)=4^{-(i+1)}(1+q_i)$ (see Figure~\ref{fig1}).

\begin{figure}[h]
\setlength{\unitlength}{1mm}
\begin{picture}(130,50)
\thicklines
\put(50,25){\circle{50}}
\put(50,25){\circle{10}}
\put(40,25){\circle{6}}
\put(33,25){\circle{4}}
\put(28,25){\circle{2}}
\put(87,25){\makebox(0,0){$(0,0,\dots,0)$}}

\put(50,17){\makebox(0,0){$c_1$}}
\put(40,19){\makebox(0,0){$c_2$}}
\put(33,20){\makebox(0,0){\small $c_3$}}
\put(28,21){\makebox(0,0){\scriptsize $c_4$}}

\drawline(25,25)(75,25)

\put(75,25){\circle*{1}}

\end{picture}
\caption{$\theta(q)$}\label{fig1}
\end{figure}
The set $\theta(q)$ is the union of  disjoint circles $c_i=S(a_i;r_i(q))$ contained in  $[-1,0]\times [-1,1]\times \{(\underbrace{0,\dots, 0}_{n-2})\}$ and of the diameter segment of the largest circle . The inner circles $c_i$ uniquely code the point $(q_i)$ and the map $\theta$ is a continuous embedding.

We also will  exploit a continuous deformation $H_0:2^{I^n}\times I\to 2^{I^n}$ through finite sets such that, for any $(A,t)\in 2^{I^n}\times (0,\frac12]$,  $H_0(A,t)$ is finite,
$$dist(A,H_0(A,t))\le 2t \quad  \text{and}\quad  H_0(A,t)\subset [t,1-t]^n.$$
Connecting points of $H_0(A,t)$, $t>0$,  one can define
another  deformation through finite graphs
\begin{equation}\label{eq:graphs}
H(A,t) = \begin{cases} \bigcup_{a,b\in H_0(A,t)} (\overline{ab}\cap (\overline{B}(a; 2t)\cup  \overline{B}(b; 2t)))& \text{if $t>0$}, \\ A& \text{if $t=0$},
\end{cases}
\end{equation}
where $\overline{B}(a; \alpha)$ is the closed $\alpha$-ball in $I^n$ around $a$ and $\overline{ab}$ is the line segment in $I^n$ from $a$ to $b$. For any  $(A,t)\in C(I^n)\times (0,\frac12)$,  $H(A,t)$ is a connected graph in $[t,1-t]^n$  and $dist(A,H(A,t))\le 4t$  (see~\cite{GM} and~\cite{Sam1} for details).

Assume now that, for each subset $M\subset I^\omega$ which belongs to class $\mathcal M$, there exists a continuous map $\xi: I^\omega \to C(I^n) (2^{I^n})$ such that $\xi^{-1}(\mathcal A)=M$ and $(0,0,\dots)\in \xi(x)$ for every $x\in I^\omega$.

An embedding  $g$ can be defined in the form
\begin{multline}\label{eq:g}
g(q)=  H(f(q),\mu(q))\  \cup \\
\bigcup_{x\in H_0(f(q),\mu(q))}(x+\mu(q)\theta(q))\  \cup \   \bigcup_{x\in H_0(f(q),\mu(q))}(x+\mu(q)\xi(q))
\end{multline}
(we use linear operations of addition and scalar multiplication in~\eqref{eq:g}), where
$$\mu(q) =\frac1{12}\min\{\epsilon, \min\{dist(f(q),f(z)): z\in K\}\}.$$
If one is interested  in the strong $\mathcal M$-universality of $\mathcal A$ in $2^{I^n}$ then the ``graph'' part $H(f(q),\mu(q))$ above is skipped.
In Section~\ref{apo}, in the proof of Theorem~\ref{t:Apo3} we modify $g$ by replacing $H(f(q),\mu(q))$ with the closed ball $\overline B\bigl(H(f(q),\mu(q)); \frac18\mu(q)\bigr)$.

One can check that $g$ is always a $Z$-embedding which agrees with $f$ on $K$ and $\epsilon$-approximates $f$  (see the proof of~\cite[Lemma 3.2]{Sam1}; some details of it are sketched  in the proof of Theorem~\ref{t:Apo3}). So, it only remains to verify property~\eqref{eq:uni}.

\

\subsection{Approach II}\label{ap2} Assume $X$ is a nondegenerate locally connected continuum without free arcs. As before, we are looking for a $Z$-embedding $g$ that approximates $f$, agrees with $f$ on $K$ and satisfies~\eqref{eq:uni}, where the cube $I^n$ is replaced with $X$.  Suppose that for each non-empty open subset $U$ of $X$ there is a continuous mapping $\varphi_U: I^\omega \to C(U)$ such that $\varphi_U^{-1}(\mathcal A)=M$.
If $X$ is a Peano continuum then each ${ \varphi }_U$ can be extended to a continuous mapping $\tilde{ \varphi }_U: I^\omega \times I \to C(U)$ such that $\tilde{\varphi}_U(q,1)={\varphi}_U(q)$ and $\tilde{\varphi}_U(q,t)$ is a finite union of arcs in $U$ for all $t\in (0,1]$ and $q\in I^\omega $.
If $U$ contains a homeomorphic image of $(0,1)^n$, $n\geq 2$, then we can put $\tilde{ \varphi }_U(q,t)=\zeta_U (H(\varphi_U(q),1-t))$ where $\zeta_U :(0,1)^n \to U$ is an embedding and $H$ is the homotopy defined in~\eqref{eq:graphs}. In this case $\tilde{ \varphi }_U(q,t)$ is a finite graph whenever $t\in (0,1]$.

Without repeating complicated details of the construction in~\cite[Section 5]{C3}, it is enough for our purposes to remark  that  the image $g(q)$, for $q\in I^\omega\setminus K$, is  of the form  $A(q)\cup \bigcup_{U\in S(q)}\overline{ \varphi }_U(q)$, where $A(q)$ is a compactum in $X$, $S(q)$ is a finite set, $\overline{ \varphi }_U(q)=\tilde{ \varphi }_U(q,t(q))$ and $\overline{ \varphi }_U(q)={ \varphi }_U(q)$ for at least one $U\in S(q)$. We can also assume that distinct $U, U'\in S(q)$ are disjoint.
In the case $\mathcal A\subset 2^X$,  $A(q)$ is at most 1-dimensional. If  $\mathcal A\subset C(X)$ then $g(q)$ is a continuum and $A(q)$ is a 1-dimensional locally connected continuum which is a countable  union of  arcs and each $\overline{\varphi}_U(q)$, $U\in S(q)$, intersects $A(q)$. Moreover, if each non-empty open subset of $X$ has dimension $\ge 2$, then we can also assume that  the intersection $\overline{\varphi}_U(q)\cap A(q)$ is a singleton for each $U\in S(q)$.
The construction guarantees that $g$ is a $Z$-embedding.

We will use this approach and verify~\eqref{eq:uni} in Sections~\ref{sep} and~\ref{WID}, and partly in~\ref{apo}.

\section{Colocally connected, aposyndetic and Kelley continua}\label{apo}

 Recall that a continuum $X$ is \emph{aposyndetic} if and only if $X$ is semi-locally connected, i.e. for any $\epsilon>0$, each point  $x\in X$ has an open neighborhood $U$  of diameter  $\diam U<\epsilon$  such that $X\setminus U$ has finitely many components; if one requires that  $X\setminus U$ be connected then $X$ is \emph{colocally connected}.

 \begin{proposition} \label{p:Apo1}
 If $Y$ is a compact space, then the families $Apo(Y)$ and $Col(Y)$ of nondegenerate aposyndetic continua and colocally connected continua in $Y$, resp., are $F_{\sigma\delta}$-subsets of $C(Y)$.
 If $Y$ contains a copy of $I^2$, then the families are $F_{\sigma\delta}$-universal.
 \end{proposition}

 \begin{proof}
 We first evaluate the Borel class of $Col(Y)$. By compactness, $X\in Col(Y)$ if and only if for each $\epsilon >0$ there is a finite $\epsilon$-cover of $X$ consisting  of open subsets of $X$
with connected complements. Passing to complements, this  can be written in terms of closed subsets of $Y$ as follows:
 \begin{multline}\label{eq:Col1}
X\in Col(Y)\quad\text{if and only if}\\
|X|>1\quad\wedge\quad (\forall n)\ (\exists m)\ (\exists K_1, \dots, K_m\in C(Y)) \\
\bigcap_{i=1}^m K_i =\emptyset\quad \wedge\quad \ (\forall i\le m) \left( K_i\subset X\quad \wedge\quad \diam (X\setminus K_i)< \frac1n\right).
  \end{multline}
Since  formula~\eqref{eq:Col1} yields merely analyticity of $Col(Y)$, it needs  further refinements.
It seems more convenient to deal with the complement of $Col(Y)$ in  $C(Y)$. Recall that the function
$$f:C(Y)\times C(Y)\to 2^Y, \quad f(X,K)=\overline{X\setminus K}$$
is lower semi-continuous~\cite[p. 182]{Ku1}. Hence, the  function
$$(X,K) \mapsto \diam \bigl(f(X,K)\bigr)$$
 is of the first Borel class~\cite[Theorem 1, p. 70]{Ku} which yields that the set
\begin{multline*}
\{\, (X,K)\in C(Y)\times C(Y): \diam (X\setminus K)\ge \frac1n\,\} = \\
\{\, (X,K)\in C(Y)\times C(Y): \diam (\overline{X\setminus K})\ge \frac1n\,\}
\end{multline*}
  is $G_\delta$ in  $C(Y)\times C(Y) $ for each $n$.  Since
$X\in C(Y)\setminus Col(Y)$  if and only if
\begin{multline}\label{eq:Col3}
|X|=1\quad \vee\quad (\exists n)\ (\forall m)\ (\forall   K_1, \dots, K_m\in C(Y))\\
\bigcap_{i=1}^m K_i \neq \emptyset\quad \vee\quad  (\exists i\le m) \left( K_i\not\subset X \quad\vee\quad \diam (X\setminus K_i)\ge \frac1n\right),
\end{multline}
we get that  $C(Y)\setminus Col(Y)$ is $G_{\delta\sigma}$ in $C(Y)$.

The proof for the family $Apo(Y)$ is similar.

Passing to the second part, we can assume that $Y$ contains $I^2$. Let $\bigl(J_{ij}\bigr)$, $i,j=1,2,\dots$, be a double sequence of mutually disjoint nondegenarate closed intervals in $I$ such that the length of $J_{ij}$ is less than $\frac1{4^{i+j}}$ and, for each $j$,  $\Lim_i J_{ij}\to \{1\}$.
Consider rectangles
$$R_{ij}(t)= J_{ij}\times  \left[0,\frac{t}{j+1}\right],$$
and  define a continuous embedding $\psi: I^\omega \to C(Y)$ by
\begin{equation}\label{eq:Col5}
\psi((q_i))=\partial(I^2)\, \cup\, \bigcup_{i,j}\,R_{ij}(q_i)
\end{equation}
(Figure~\ref{fig4}; instead of rectangles one can also use their boundaries).
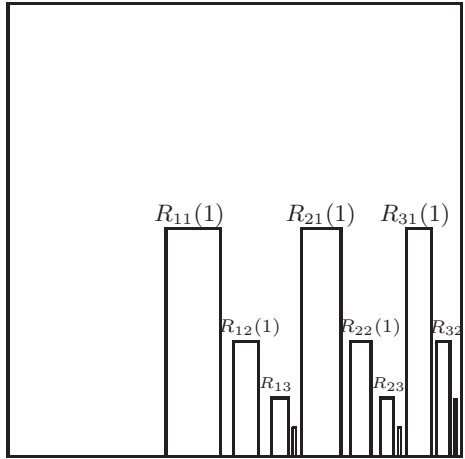
\begin{figure}[h]
\setlength{\unitlength}{1mm}
\begin{picture}(100,60)
\thicklines

\put(25,0){\framebox(60,60)}
\put(49,32){\makebox(0,0){\small$R_{11}(1)$}}
\put(46,0){\framebox(7,30)}

\put(55,0){\framebox(3,15)}
\put(57,17){\makebox(0,0){\scriptsize$R_{12}(1)$}}

\put(60,0){\framebox(2,7.5)}
\put(60.5,9.5){\makebox(0,0){\tiny$R_{13}$}}

\put(73,17){\makebox(0,0){\scriptsize$R_{22}(1)$}}
\put(70.5,0){\framebox(2.5,15)}

\put(75.5,9.5){\makebox(0,0){\tiny$R_{23}$}}
\put(74.5,0){\framebox(1.5,7.5)}

\put(66.5,32){\makebox(0,0){\small$R_{21}(1)$}}
\put(64,0){\framebox(5,30)}

\put(79,32){\makebox(0,0){\small$R_{31}(1)$}}
\put(78,0){\framebox(3,30)}

\put(83.2,17){\makebox(0,0){\scriptsize$R_{32}$}}
\put(82,0){\framebox(1.5,15)}

\thinlines
\put(62.7,0){\line(0,1){3.7}}
\put(63.2,0){\line(0,1){3.7}}
\put(62.7,3.7){\line(1,0){.5}}

\put(76.7,0){\line(0,1){3.7}}
\put(77.1,0){\line(0,1){3.7}}
\put(76.7,3.7){\line(1,0){.4}}

\put(84.2,0){\line(0,1){7.5}}
\put(84.5,0){\line(0,1){7.5}}
\put(84.2,7.5){\line(1,0){.3}}

\end{picture}
\caption{$\psi\bigl(1,1,1,\dots)$}\label{fig4}
\end{figure}

It satisfies
\begin{equation}\label{eq:Col4}
\psi^{-1}(Col(Y))= \psi^{-1}(Apo(Y))=\widehat{c_0}
\end{equation}
and since  $\widehat{c_0}$ is an $F_{\sigma\delta}$-absorber in $I^\omega$, the proof is complete.

\end{proof}

\begin{proposition}\label{p:Apo2}
If  $ D(Y)$ denotes the family of all decomposable continua in a space $Y$, then $Col(Y)\subset Apo(Y)\subset D(Y)$ and $D(I^n)$ is a $\sigma Z$-set in $C(I^n)$ for $n\ge 3$.
\end{proposition}
\begin{proof}
It is known that each nondegenerate aposyndetic continuum is decomposable. The last part of the proposition was proved in~\cite{Sam1}.

\end{proof}

\begin{theorem}\label{t:Apo3}
$Apo(I^n)$ and $Col(I^n)$ are $F_{\sigma\delta}$-absorbers in $C(I^n)$ for $n\ge 3$.
\end{theorem}

\begin{proof}
Concerning $Apo(I^n)$, we can refer to the proof of the  strong  $F_{\sigma\delta}$-universa\-lity of the family of all Peano continua  $LC(I^n)$ due to Gladdines and van Mill~\cite{GM}. It perfectly works for $Apo(I^n)$. Below, we sketch an alternative  construction that suits both families. Actually, we appropriately modify the construction of the embedding $g$ from Subsection~\ref{ap1}.

First, locate continua $\psi((q_i))$ from~\eqref{eq:Col5} in  $I^n$ as
$$
\psi'((q_i))=\psi((q_i))\times \{(\underbrace{0,\dots, 0}_{n-2})\}.
$$
The set $\widehat{c_0}$ being an $F_{\sigma\delta}$-absorber in $I^\omega$, there is, for an $F_{\sigma\delta}$-set $M\subset I^\omega$, a mapping $\zeta:I^\omega\to I^\omega$ such that $\zeta^{-1}(\widehat{c_0})=M$. Put
\begin{equation}\label{xi}
\xi=\psi'\zeta.
\end{equation}
Second, since the ``graph" ingredient $H(f(q),\mu(q))$ may spoil the colocal connectedness of $g(q)$, we surround it by a closed ball $$\overline B\bigl(H(f(q),\mu(q)); \frac18\mu(q)\bigr)$$ of a small enough radius (e. g., $\frac18\mu(q)$) and redefine the embedding as follows:

\begin{multline}\label{eq:g'}
g'(q)=  \overline B\bigl(H(f(q),\mu(q)); \frac18\mu(q)\bigr)\  \cup \\
\bigcup_{x\in H_0(f(q),\mu(q))}(x+\mu(q)\theta(q))\  \cup \   \bigcup_{x\in H_0(f(q),\mu(q))}(x+\mu(q)\xi(q)).
\end{multline}

Clearly, $g'$ is an $\epsilon$-approximation of $f$ and $g'|K=f|K$. So, it is 1-to-1 on $K$.  For distinct $q,q'\in I^\omega\setminus K$, coefficients  $\mu(q)$ and $\mu(q')$ are positive. Let $x$, $x'$ be
 the minimal points in $H_0(f(q),\mu(q))$ and $H_0(f(q'),\mu(q'))$, respectively, in the lexicographic order on $I^n$.  Then $g'(q)\neq g'(q')$ because the inner circles in copies  $x+\mu(q)\theta(q)$ and $x'+\mu(q')\theta(q')$ remain disjoint from the rest of $g'(q)$ and $g'(q')$, respectively. If $q\in K$, $q'\notin K$, then $g'(q)\neq g'(q')$ by a similar simple estimation of $dist(g'(q), g'(q'))$  as for the original approximation $g$.
  Moreover, $g'(K)=f(K)$ is a $Z$-set in $C(I^n)$ by assumption and $g'(I^\omega)\setminus g'(K)$, as an open subset of $g'(I^\omega)$, is  $F_\sigma$ in $C(I^n)$. Notice that  $g'(q)$, for each $q\in I^\omega\setminus K$, contains  an open in $g'(q)$, one-dimensional subset of    $x+\mu(q)\theta(q)$ (e. g., the inner circles in there). Therefore, the deformation
  $$h:C(I^n)\times I\to C(I^n),\quad h(A,t)=\overline B(A;t)$$
    maps $g'(I^\omega\setminus K)$ off the set $I^\omega\setminus K$ arbitrarily closely to the identity map on $C(I^n)$. It means that  $g'(I^\omega)\setminus g'(K)$  is a $\sigma Z$-set and, consequently,  $g'(I^\omega)$ is a $Z$-set in  $C(I^n)$.

Finally, we need to check the property
\begin{equation}\label{eq:red}
 (g')^{-1}(Col(I^n))\setminus K= (g')^{-1}(Apo(I^n))\setminus K=M\setminus K.
\end{equation}

So,  suppose $q\notin K$. Then $\mu(q)>0$. If $q\in M$, then one can easily see that $g'(q)\in Col(I^n)$. If  $q\notin M$, it is convenient to consider  the maximal point $y=(y_1,y_2,\dots,y_n)  \in  H_0(f(q),\mu(q))$ in the lexicographic order $\prec$ on $I^n$.  The copy $y+\mu(q)\xi(q)$ of $\psi(\zeta(q))$  is not aposyndetic by~\eqref{eq:Col4} and it can be intersected by at most finitely many other isometric copies  $x+\mu(q)\xi(q)$, $x\in H_0(f(q),\mu(q))$, where $x=(x_1,x_2,y_3,\dots,y_n)\prec y$. Neither these copies nor adding the part
$$\overline B\bigl(H(f(q),\mu(q)); \frac18\mu(q)\bigr)\,  \cup\,
\bigcup_{x\in H_0(f(q),\mu(q))}(x+\mu(q)\theta(q))$$
affect the non-semi-local connectedness  caused by  $y+\mu(q)\xi(q)$. So, $g'(q)\notin Apo(I^n)$ and~\eqref{eq:red} is satisfied.

\end{proof}

Recall that a continuum $Y$ is called a \emph{Kelley continuum}  if  for each point $z\in Y$, each sequence of points $z_n\in Y$ converging to $z$ and each subcontinuum $Z\subset Y$ such that $z\in Z$, there is a sequence of subcontinua $Z_n\subset Y$, $z_n\in Z_n$, that converge to $Z$ (in the sense of the Hausdorff distance).

One can easily observe that the  proof of the strong $F_{\sigma\delta}$-universality of $Col(I^n)$ presented above applies directly to family $\mathcal K(I^n)$, i.e., \eqref{eq:red} is satisfied if $Col(I^n)$ is substituted with $\mathcal K(I^n)$.  Moreover, if $X$ is a locally connected continuum each of whose non-empty open subset contains a copy of $I^n$, $n\ge 2$, then $\mathcal K(X)$ is strongly $F_{\sigma\delta}$-universal in $C(X)$ because Approach II~\ref{ap2} applies to $\mathcal K(X)$ with mappings $\varphi_U$ being compositions of $\xi$~\eqref{xi} with embeddings $I^n\hookrightarrow U$. Exactly the same observation concerns  families $Apo(X)$ and $LC(X)$ of locally connected subcontinua of $X$.     We do not know if $\mathcal K(I^n)$ is contained in a $\sigma Z$-set in $C(I^n)$ but if we restrict the family to  decomposable continua, then the condition is satisfied for $n\ge 3$, since  $D(I^n)$ is a $\sigma Z$-set in $C(I^n)$  (Proposition~\ref{p:Apo2}).  Concerning the more general case of a Peano continuum $X$ as above, it is not known if decomposable subcontinua of $X$ form a  $\sigma Z$-set in $C(X)$, so, instead, we can restrict  $\mathcal K(X)$ (and $Apo(X)$ and $LC(X)$) to  $\mathcal D_2(X)$ which is a $\sigma Z$-set in $C(X)$ by Example~\ref{ex1}(1).
Summarizing, we get the following theorem.

\begin{theorem} Let $X$ be a locally connected continuum each of whose non-empty open subset contains a copy of $I^n$, $n\ge 2$. Then
 $\mathcal K(X)$, $Apo(X)$ and $LC(X)$ are strongly $F_{\sigma\delta}$-universal in $C(X)$. The families  $\mathcal K(X)\cap \mathcal D_2(X)$,   $Apo(X)\cap \mathcal D_2(X)$, $LC(X)\cap \mathcal D_2(X)$ are $F_{\sigma\delta}$-absorbers in $C(X)$ and
 $\mathcal K(I^n)\cap D(I^n)$ is an $F_{\sigma\delta}$-absorber in $C(I^n)$ for $n\ge 3$.
 \end{theorem}

\begin{remark}
The more general Approach II~\ref{ap2} cannot be directly applied for $Col(X)$ with  mappings $\varphi_U$ taken as copies of the reduction $\xi$~\eqref{xi}, since the 1-dimensional part $A(q)$ of $g(q)$ can spoil the colocal connectedness of $g(q)$ and the remedy of surrunding it by a small closed ball (similarly as in the proof of Theorem~\ref{t:Apo3})  may kill the ``one-to one'' property of  $g(q)$.
\end{remark}

\section{Closed separators of $I^n$}\label{sep}
\begin{proposition}\label{p:sep}
If $X$ is a locally connected continuum, then the family $\mathcal S(X)$ of all compact separators of $X$ is an $F_\sigma$-subset of $2^X$.
\end{proposition}
\begin{proof}
Since $X$ is locally connected, a closed subset $S$ separates $X$ if and only if $S$ cuts $X$ between two points. Let
$\mathcal E$ be a countable dense subset of the family  $\{(C,x,y)\in C(X)\times X^2: x\neq y, \  x,y \in C\}$. Denote by $proj_1$ and $proj_2$ the projections of $C(X)\times X^2$ onto $C(X)$ and $X^2$, respectively.
We can now express the definition of a closed separator (= cut) using $\mathcal E$:
\begin{quote}
$S$ is a closed separator of $X$ if and only if
\begin{multline}\label{e:sep}
\exists (x,y)\in proj_2(\mathcal E) \quad (\{x,y\}\cap S=\emptyset) \quad\text{and}\\
\forall C\in proj_1(\mathcal E\cap (proj_2)^{-1}(x,y)) \quad (C\cap S \neq\emptyset).
\end{multline}
\end{quote}
Since the two quantifiers in~\eqref{e:sep} are taken over countable sets, the conclusion follows.

\end{proof}

\begin{proposition}\label{p1:sep}
If a space $X$ contains an open subset homeomorphic to the combinatorial interior $\inte (I^n) =(0,1)^n$ of $I^n$, $2\le n<\infty$, then  $\mathcal S(X)\cap C(X)$ is $F_\sigma$-universal.
\end{proposition}

\begin{proof}
We can assume,  without loss of generality, that $I^n\subset X$ and $(0,1)^n$ is open in $X$.
Since the pseudo-boundary $B(I^\omega)$ is an $F_\sigma$-absorber, it is enough to construct an embedding
$\Psi:I^\omega\to C(X)$ such that
\begin{equation}\label{e:psi}
\Psi((q_i))\in \mathcal S(X)\quad\text{if and only if}\quad (q_i)\in B(I^\omega).
\end{equation}

 Denote  $J_i=[\frac1{2i+1},\frac1{2i}]$ and let $\partial(\Delta_i)$ be the combinatorial boundary of the cube $$\Delta_i=J_i \times \left[\frac12,\frac12+\frac1{4i(i+1)}\right]^{n-1}, \quad i=1,2,\dots.$$
 For each $i$, there is a deformation $h_i(A,t)$ of $2^{\Delta_i}$ through finite sets in
 $$\left[\frac1{2i+1}+t,\frac1{2i}-t\right]\times \left[\frac12+t,\frac12+\frac1{4i(i+1)}-t\right]^{n-1},\  0<t<\frac1{8i(i+1)}$$
 (look at $h_i$ as $H_0$ considered in Section~\ref{su} with $I^n$ replaced by $\Delta_i$).
 Given  $t\in I$, choose points
 $$x_i(t)=\left(\frac1{2i+1}+t\frac1{4i(i+1)},\frac12,\dots,\frac12\right)$$ and
 $$y_i(t)=\left(\frac1{2i+1}+t\frac1{4i(i+1)},\frac12-\frac1{4i(i+1)},\frac12,\dots,\frac12\right).$$

 For $0<t<\frac1{8i(i+1)}$ and each $i$, connect points of the set  $h_i(A,t)\cup \{x_i\}$ following formula~\eqref{eq:graphs}:
\begin{equation}\label{eq:H_i}
H_i(A,t)=  \bigcup_{a,b\in h_i(A,t)\cup \{x_i\}} (\overline{ab}\cap (\overline{B}(a; 2t)\cup  \overline{B}(b; 2t)))
\end{equation}
and put $H_i(A,0)=A$.
Thus $H_i$ is a deformation $2^{\Delta_i}$ through finite graphs in $\Delta_i$ that meet the edge $J_i\times \{(\frac12,\dots,\frac12)\}$ at the single point $x_i$.
\begin{figure}[h]
\setlength{\unitlength}{1mm}
\begin{picture}(135,30)
\thicklines
\put(0,0){\line(1,0){120}}
\put(40,0){\framebox(20,20)}
\put(40,0){\line(0,-1){20}}
\put(30,0){\line(0,-1){6}}
\put(17,0){\line(0,-1){3}}

\put(108,-4){\makebox(0,0){$I\times \left(\frac12,\frac12,\dots\right)$}}

\put(60,20){\line(1,1){6}}
\put(60,0){\line(1,1){6}}
\put(66,6){\line(0,1){20}}
\put(46,26){\line(1,0){20}}
\put(40,20){\line(1,1){6}}
\put(24,0){\framebox(6,6)}
\put(17,0){\framebox(3,3)}

\put(45,-3){\makebox(0,0){\scriptsize$x_1(q_1)$}}
\put(45,-21){\makebox(0,0){\scriptsize$y_1(q_1)$}}
\put(35,-2){\makebox(0,0){\scriptsize$x_2(q_2)$}}
\put(35,-6){\makebox(0,0){\scriptsize$y_2(q_2)$}}

\put(50,30){\makebox(0,0){\small$\partial(\Delta_1)$}}

\put(30,12){\makebox(0,0){\scriptsize$\partial(\Delta_2)$}}

\put(19,7){\makebox(0,0){\scriptsize$\partial(\Delta_3)$}}

\put(7,3){\makebox(0,0){$\cdots$}}

\put(30,6){\line(1,1){4}}
\put(30,0){\line(1,1){4}}
\put(34,4){\line(0,1){6}}
\put(27.5,10){\line(1,0){7}}
\put(24,6){\line(1,1){4}}

\put(17,3){\line(1,1){2}}
\put(20,3){\line(1,1){2}}
\put(20,0){\line(1,1){2}}
\put(22,5){\line(0,-1){3}}
\put(19,5){\line(1,0){3}}
\end{picture}
\vspace{2cm}
\caption{ $\Psi\bigl((0,1,0,\dots)\bigr)$}\label{fig2}
\end{figure}
The function $\Psi:I^\omega\to  C(I^n)$ defined by
\begin{multline}\label{eq:w}
 \Psi\bigl((q_i)\bigr)=\\ I\times \left\{\left(\frac12,\dots,\frac12\right)\right\}\ \cup \ \bigcup_i H_i\left(\partial(\Delta_i),\frac{q_i(1-q_i)}{8i(i+1)}\right)\cup \bigcup_i \overline{x_i(q_i)y_i(q_i)}
 \end{multline}
 (see Figure~\ref{fig2}) is  continuous and it is 1-1 since the correspondence $(q_i)\mapsto (\overline{x_i(q_i)y_i(q_i)})$ is 1-1.
It satisfies~\eqref{e:psi} for $n\ge 3$ because if $(q_i)\in B(I^\omega)$, then  $\Psi\bigl((q_i)\bigr)$ contains $\partial(\Delta_i)$ for some $i$ which separates $\inte(I^n)$ --- consequently, by construction, $\Psi\bigl((q_i)\bigr)$ separates  $X$; otherwise,  $\Psi\bigl((q_i)\bigr)$ is one-dimensional, so it does not separate $I^n$, $n\ge 3$.

 If $n=2$, the above construction does not work, since $\Psi\bigl((q_i)\bigr)$ might separate $I^2$ for $(q_i)$ in the pseudo-interior $s$. In this case we can, however, define an appropriate  embedding  $\Psi$ differently much easier. Given $(q_i)\in I^\omega$,  denote $p_{2i}=x_i(q_i)$,  $p_{2i-1}=x_i(1-q_i)$, $d_i=\bigl(\frac1{2i+1},\frac12\bigr)$, $A_i=\overline{d_ip_i}\setminus \{d_i,p_i\}$.
Put
\begin{equation}\label{Psi2}
\Psi\bigl((q_i)\bigr)= I\times \left\{\frac12\right\}\  \cup \ \bigcup_i \left(\partial(\Delta_i)\setminus A_i\right).
\end{equation}
\end{proof}
\begin{theorem}\label{t:sep}
Let $X$ be a locally connected continuum such that each open non-empty subset of $X$ contains a copy of $(0,1)^n$, $3\le n<\infty$, as an open subset and no subset of dimension $\le 1$ separates $X$. Then the families $\mathcal S(X)$ and  $\mathcal S(X)\cap C(X)$ are $F_\sigma$-absorbers in $2^{X}$ and $C(X)$, respectively.
\end{theorem}
\begin{proof}
Since being an $F_\sigma$-absorber in a Hilbert cube  is equivalent, for an $F_\sigma$-set,  to being strongly $F_\sigma$-universal~\cite[Theorem 5.3]{BGM},
it remains to prove the strong $F_\sigma$-universality. To this end, we are going  to use Approach II~\ref{ap2}. The pseudo-boundary
 $B(I^\omega)$ being strongly $F_\sigma$-universal, there exists for each $M\subset I^\omega$, $M\in F_\sigma$, a mapping $\chi: I^\omega \to I^\omega$ such that
$\chi^{-1}(B(I^\omega))=M$. Hence, the composition $\varphi=\Psi\chi:I^\omega \to C(I^n)$ ($\Psi$ as in~\eqref{eq:w}) satisfies
$\varphi^{-1}(\mathcal S(I^n))=M$ for every $q\in I^\omega$.  For each open non-empty subset $U$ of $X$ and an open copy of $(0,1)^n$ in $U$, let $\varphi_U:I^\omega \to C(U)$ be a composition of $\varphi$ with an embedding of $C(I^n)$ into the hyperspace of that copy. Then $\varphi_U(q)$ separates $X$ iff $\varphi(q)$ separates $I^n$ iff $q\in B(I^\omega)$.

 The properties of $g$ in Approach II~\ref{ap2} yield that  $g$ satisfies~\eqref{eq:uni} for $\mathcal A=S(X)$ as well as for $\mathcal A=S(X)\cap C(X)$. In fact, for $q\notin K$,   $g\bigl(q\bigr)$ separates $X$ if and only if $\varphi_U(q)$ does, because neither the zero- nor one-dimensional part $A(q)$ of $g\bigl(q\bigr)$ do  affect separation of $X$ by  copies $\varphi_U(q)$ for $n\ge 3$.
This completes the proof.
\end{proof}

\begin{corollary}
The families $\mathcal S(X)$ and  $\mathcal S(X)\cap C(X)$ are $F_\sigma$-absorbers in $2^{X}$ and $C(X)$, respectively, if $X$ is a continuum which is an $n$-manifold (with or without boundary), $3\le n <\infty$.
\end{corollary}

Denote by $\mathcal N(X)$  the family of all nowhere dense closed subsets of $X$.
The following proposition is well known and has a straightforward proof.
\begin{proposition}\label{p:nd}
For any compact space $X$, the subspace of  $2^X$ consisting of all closed subsets of $X$ with non-empty interiors is an $F_\sigma$-set.
\end{proposition}
The next fact follows from Propositions~\ref{p:sep} and~\ref{p:nd}.
\begin{proposition}\label{p:ndsep}
If $X$ is a locally connected continuum, then
$$\mathcal S(X)\cap \mathcal N(X)\in D_2(F_\sigma) \quad\text{and}\quad  \mathcal S(X)\cap \mathcal N(X)\cap C(X)\in D_2(F_\sigma).$$
\end{proposition}
\begin{proposition}\label{p1:ndsep}
If a space $X$ contains an open subset homeomorphic to the combinatorial interior $\inte (I^n) =(0,1)^n$ of $I^n$, $2\le n<\infty$, then  $\mathcal S(X)\cap \mathcal N(X)\cap C(X)$ is $D_2(F_\sigma)$-universal.
\end{proposition}
\begin{proof}
Assume again that $I^n\subset X$ and $(0,1)^n$ is an open subset of $X$.  First, replace in the definition~\eqref{eq:w} of embedding $\Psi$ the boundaries $\partial(\Delta_i)$ with the cubes $\Delta_i$ themselves for all $i$'s and denote thus obtained embedding by $\Psi_0$. We have
\begin{equation}\label{e:phi}
\Psi_0:I^\omega\to C(I^n)\quad\text{and}\quad \Psi_0^{-1}\bigl(\mathcal N(X)\bigr)=s,
\end{equation}
where $s$ is the pseudo-interior of $I^\omega$.

Let $\alpha:I^n \to [-1,0]\times I^{n-1}$ be the reflection
\begin{equation}\label{alpha}
\alpha(x_1,x_2,\dots, x_n)=(-x_1,x_2,\dots, x_n).
\end{equation}
Fix an embedding $f:[-1,1]\times I^{n-1}\to \inte(I^n)\subset X$.
Define a continuous embedding
\begin{equation}\label{Phi}
\Phi: I^\omega\times I^\omega \to C(X), \quad  \Phi((q_i), (t_i)\bigr)=f\bigl(\Psi\bigl((q_i)\bigr)\, \cup\,  \alpha\bigl(\Psi_0\bigl((t_i)\bigr)\bigr)\bigr).
\end{equation}
For $n\ge 3$, it follows from~\eqref{e:psi} and~\eqref{e:phi} that
\begin{multline}\label{Phireduction}
\Phi((q_i), (t_i)\bigr) \in  \mathcal S(X)\cap \mathcal N(X)\cap C(X) \quad\text{if and only if}\\
 ((q_i), (t_i)\bigr) \in B(I^\omega)\times s.
 \end{multline}
 In case $n=2$ we define  map $\Phi$ by  formula~\eqref{Phi} in which $\Psi$ is the map from~\eqref{Psi2} and $\phi$ is defined in the following way.

 Let $f_i:2^{J_i}\times I\to 2^{J_i}$ be a deformation through finite sets. Denote
 \begin{multline*}
 W_i(t)=f_i(J_i,t)\times \left[\frac12,\frac12+\frac1{4i(i+1)}\right], \\
 V_{2i-1}= W_i\bigl((q_i\bigr)), \quad V_{2i}= W_{i}\bigl((1-q_i\bigr))
 \end{multline*}
 for $(q_i)\in I^\omega$.
 Observe that $W_i(0)=\Delta_i$. Now, the map
 $$\phi\bigl((q_i\bigr))= I\times \left\{\frac12\right\}\,\cup\,\bigcup_i\, V_i$$
 is an embedding which satisfies~\eqref{Phireduction} (Figure~\ref{fig3}).
 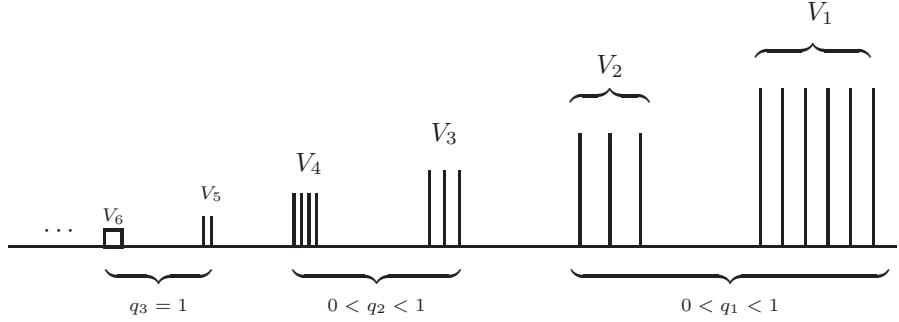
\begin{figure}[h]
\setlength{\unitlength}{1mm}
\begin{picture}(135,30)
\thicklines
\put(0,0){\line(1,0){118}}


\put(108,31){\makebox(0,0){$V_1$}}

\put(107,26){\makebox(0,0){$\overbrace{\ \ \ \ \ \ \ \ \ \ \ \ \  }$}}
\put(96,-4){\makebox(0,0){$\underbrace{\ \ \ \ \ \ \ \ \ \ \ \ \ \ \ \ \ \ \ \ \ \ \ \ \ \ \ \ \ \ \ \ \ \ \ \  }$}}
\put(49,-4){\makebox(0,0){$\underbrace{\ \ \ \ \ \ \ \ \ \ \ \ \ \ \ \ \ \ \   }$}}

\put(96,-8){\makebox(0,0){\scriptsize$0<q_1<1$}}

\multiput(100,0)(3,0){6}{\line(0,1){21}}

\put(80,24){\makebox(0,0){$V_2$}}

\put(80,20){\makebox(0,0){$\overbrace{\ \ \ \ \ \ \ \ \    }$}}

\multiput(76,0)(4,0){3}{\line(0,1){15}}

\multiput(56,0)(2,0){3}{\line(0,1){10}}

\put(58,15){\makebox(0,0){$V_3$}}

\put(49,-8){\makebox(0,0){\scriptsize$0<q_2<1$}}

\multiput(38,0)(1,0){4}{\line(0,1){7}}

\put(40,11){\makebox(0,0){$V_4$}}

\multiput(26,0)(1,0){2}{\line(0,1){4}}

\put(13,0){\framebox(2,2)}

\put(20,-8){\makebox(0,0){\scriptsize$q_3=1$}}
\put(27,7){\makebox(0,0){\scriptsize$V_5$}}
\put(20,-4){\makebox(0,0){$\underbrace{\ \ \ \ \ \ \ \ \ \ \ \   }$}}

\put(14,4){\makebox(0,0){\scriptsize$V_6$}}
\put(7,2){\makebox(0,0){$\cdots$}}
\end{picture}
\vspace{.5cm}
\caption{$\phi\bigl((q_1,q_2,1,\dots)\bigr)$ in $I^2$}\label{fig3}
\end{figure}

Since $B(I^\omega)\times s$ is strongly $D_2(F_\sigma)$-universal (see Example~\ref{ex4}), the proof is complete.
\end{proof}

\begin{theorem}\label{t:ndsep}
Assume $X$ satisfies hypotheses of Theorem~\ref{t:sep}. Then  $\mathcal S(X)\cap \mathcal N(X)$ is an $D_2(F_\sigma)$-absorber in $2^{X}$ and
 $\mathcal S(X)\cap \mathcal N(X)\cap C(X)$ is an $D_2(F_\sigma)$-absorber in $C(X)$ for $n\ge3$.
\end{theorem}
\begin{proof}
Since
$$\mathcal S(X)\cap \mathcal N(X)\subset \mathcal S(X), \quad \mathcal S(X)\cap \mathcal N(X)\cap C(X)\subset\mathcal S(X) \cap C(X)$$
 and $\mathcal S(X)$ and $\mathcal S(X) \cap C(X)$ are $\sigma Z$-sets in $2^{X}$ and $C(X)$, respectively (by Theorem~\ref{t:sep}),  it remains to show the strong $D_2(F_\sigma)$-universality of  both families. There exists, for each subset $M\subset I^\omega$ from the class $D_2(F_\sigma)$, a mapping
$$\zeta:I^\omega \to I^\omega \times I^\omega \quad\text{such that}\quad \zeta^{-1}(B(I^\omega)\times s)=M.$$

The  composition
$\varphi= \Phi\zeta$, where $\Phi$ is the mapping~\eqref{Phi},
provides a map $I^\omega \to C(I^n)\subset C(X)$ satisfying
$$\varphi^{-1}\bigl((\mathcal S(X)\cap \mathcal N(X)\cap C(X)\bigr)=M$$ and mappings  $\varphi_U:I^\omega \to C(U)$, for every open non-empty subset $U$ of $X$ and an open copy of $(0,1)^n$ in $U$, which are compositions of $\varphi$ with an embedding of $C(I^n)$ into the hyperspace of that copy.
Now, the construction of the embedding $g$ in Approach~\ref{ap2} works for both families, i.e., if  $q\notin K$,  then
\begin{multline*}
g\bigl(q\bigr)\in \mathcal S(X)\cap \mathcal N(X)\cap C(X) \quad\text{if and only if}\\
 \varphi\bigl(q\bigr) \in \mathcal S(I^n)\cap \mathcal N(I^n)\cap C(I^n)
 \end{multline*}
 and similarly for $\mathcal S(X)\cap \mathcal N(X)$, because neither the zero- nor one-dimensional part $A(q)$ of $g\bigl(q\bigr)$ destroy the separation properties of  copies $\varphi_U(q)$ in $g\bigl(q\bigr)$ or change their status of being nowhere dense in $X$  for $n\ge 3$.

\end{proof}

\begin{corollary}
If $X$ is a continuum which is an $n$-manifold (with or without boundary), $3\le n <\infty$, then the family  $\mathcal S(X)_{n-1}$ of all $(n-1)$-dimensional closed separators of $X$ is a $D_2(F_\sigma)$-absorber in $2^{X}$ and $\mathcal S(X)_{n-1}\cap C(X)$ is a $D_2(F_\sigma)$-absorber in $C(X)$.
\end{corollary}
\begin{proof}
It is well known that each such $X$ is a Cantor $n$-manifold ($\equiv$ no subset of dimension $\le n-2$ separates $X$) and $Y\in \mathcal N(X)$ iff $\dim Y\le n-1$. This means that
$$\mathcal S(X)\cap \mathcal N(X) = \mathcal S(X)_{n-1}.$$
\end{proof}

\section{Infinite-dimensional compacta}\label{WID}
Recall that a space $X$ is \emph{strongly infinite-dimensional} if
there exists a sequence $(A_n, B_n)_n$ of closed disjoint subsets of $X$
such that for each  sequence $(C_n)_n$ of closed separators of $X$ between $A_n$ and $B_n$
we have $\bigcap_n C_n \neq\emptyset$.

A space is \emph{weakly infinite-dimensional} if it is not strongly infinite-dimensional.  The collection of all weakly infinite-dimen\-sio\-nal compacta in a space $Y$ will be denoted by $\mathcal W(Y)$.

\

It was proved in~\cite[Section 4, p. 173]{Pol1} that strongly infinite-dimen\-sional compacta form an analytic subset of $2^{I^\omega}$. Below, we provide a different and elementary proof of this fact for such compacta in an arbitrary locally connected compact space.

\begin{proposition}\label{p1}
The family of strongly infinite-dimensional compacta in a compact locally connected space $Y$ is an analytic subset of $2^{Y}$.
\end{proposition}

\begin{proof}
In view of Observation~\ref{o1}, we have the following claim.
\begin{claim}\label{c:cut}
A compact space $X\subset Y$ is strongly infinite-dimensional if and only if
\begin{multline}\label{e:cut}
\exists (A_n,B_n)_n\in (2^Y\times 2^Y)^\omega\\
\forall n \ (A_n\subset X, B_n\subset X, A_n\cap B_n=\emptyset)\quad
\text{and}\quad \forall (C_n)\in (2^Y)^\omega\\
(\text{if}\quad \forall n \ (\text{$C_n$ cuts  $Y$ between $A_n$ and $B_n$),
then $X\cap \bigcap_n C_n\neq\emptyset$}).
\end{multline}
\end{claim}
A rough evaluation of the projective complexity of  formula~\eqref{e:cut} gives merely class $ \Pi^1_2$. Therefore we need to refine the cutting condition in~\eqref{e:cut}.

Let $\mathcal E$ be a countable dense subset of the family
$$\{(C,A,B)\in (2^Y)^3: A\cap B=\emptyset,\quad\text{$C$ cuts  $Y$ between $A$ and $B$}\}$$ and let $\mathcal E_1=proj_1(\mathcal E)$, where $proj_1$ is the projection of $(2^Y)^3$ onto the first factor-space.

\begin{claim}\label{c:count cut}
A compact space $X\subset Y$ is strongly infinite-dimensional if and only if
\begin{multline}\label{e:count cut} \exists (A_n,B_n)_n\in (2^Y\times 2^Y)^\omega\\
\forall n \ (A_n\subset X, B_n\subset X, A_n\cap B_n=\emptyset)\quad
\text{and}\quad
 \forall k \ \forall (C_1,\dots,C_k)\in (\mathcal E_1)^k \\
 (\text{if}\quad \forall i\le k \ (\text{$C_i$ cuts  $Y$ between $A_i$ and $B_i$),
then $X\cap \bigcap_{i\le k} C_i\neq\emptyset$}).
\end{multline}
\end{claim}
In order to show the less obvious implication $\Leftarrow$, assume that $C_n$ cuts   $Y$ between $A_n$ and $B_n$ for each $n\in\mathbb N$.
For each $k$, approximate $(C_i,A_i,B_i)$ by $(C'_i,A'_i,B'_i)\in\mathcal E$, $i\le k$. Then, by the local connectedness of $Y$, $C'_i$ cuts $Y$ between $A_i$ and $B_i$ if the approximation is sufficiently close. Then $X \cap \bigcap_{i\le k} C'_i\neq\emptyset$. So, there is a point
$x_k\in X \cap \bigcap_{i\le k} C'_i$ and the distances $d(x_k,C_1),\dots, d(x_k,C_k)$ can be made arbitrarily small. Let $x$ be an accumulation point of sequence $(x_k)$. Then $x\in X\cap \bigcap_{i\le k} C_i$ and, by the compactness of $Y$, we get $X\cap \bigcap_{n\in\mathbb N} C_n\neq\emptyset$ which completes the proof of Claim~\ref{c:count cut}.

\

We are now ready to evaluate the complexity of formula~\eqref{e:count cut}. First, write down the sentence
\begin{quote}
\emph{$C_i$ cuts  $Y$ between $A_i$ and $B_i$}
\end{quote}
as
\begin{multline}
C_i\cap A_i=\emptyset,\quad C_i\cap B_i=\emptyset \quad\text{and}\quad  \forall D\in C(Y)\\
(D\cap A_i \neq\emptyset, \quad D\cap B_i \neq\emptyset)\Rightarrow D\cap C_i\neq\emptyset
\end{multline}
and observe that its Borel complexity is $G_\delta$. Next, notice that all quantifiers in formula~\eqref{e:count cut}, preceding  the sentence, except for the first existential one, are taken over at most countable sets of variables, hence the  formula following the first existential quantifier describes a Borel set. Now, the whole formula~\eqref{e:count cut} gives an analytic set, as a continuous projection of a Borel set.

\end{proof}

\begin{proposition}\label{co1}
Let $Y$ be a compact, locally connected space containing a Hilbert cube. Then,
for each integer $n\ge 0$, the subsets $\mathcal W_n(Y)$ of $2^Y$ and $\mathcal W_n(Y)\cap C(Y)$ of $C(Y)$ consisting of weakly infinite-dimensional compacta and continua, respectively,  of dimensions $\ge n$ are $\Pi^1_1$-complete.
\end{proposition}
\begin{proof}
 It is well known that the family $\mathcal D_n(Y)$ of all closed subsets of $Y$ that have dimension at least $n$ is $F_\sigma$ in $2^Y$ (for any compact $Y$). Since $\mathcal W_n(Y)= \mathcal W(Y) \cap \mathcal D_n(Y)$,  the set $\mathcal W_n(Y)$ is coanalytic  by Proposition~\ref{p1}. Hence   $\mathcal W_n(Y)\cap C(Y)$ is also coanalytic. The $\Pi^1_1$-hardness of $\mathcal W(Y)\cap C(Y)$ for $Y=I^\omega$ was established in~\cite[Corollary 3.3]{Kr} and an analogous argument gives the hardness for each set $\mathcal W_n(Y)\cap C(Y)$ and arbitrary $Y$ as in the hypothesis. It follows immediately that  each $\mathcal W_n(Y)$ is $\Pi^1_1$-hard as well.

\end{proof}

Recall that a space $X$ is  a $C$-\emph{space} if for each sequence $\mathcal U_1, \mathcal U_2, \ldots$
of open covers of $X$
there exists a sequence $\mathcal V_1, \mathcal V_2, \ldots$,
of families of pairwise disjoint open subsets of $X$ such that
each $\mathcal V_i$ refines $\mathcal U_i$ and
$\bigcup_{i=1}^\infty \mathcal V_i$ is a cover of $X$.

If $(X,d)$ is compact, then covers $\mathcal U_1, \mathcal U_2, \ldots$ can be replaced by a sequence of positive reals $\epsilon_i\rightarrow 0$, if $i\rightarrow \infty$, and  the cover $\bigcup_{i=1}^\infty \mathcal V_i$ can be replaced by a finite subcover. Each family $\mathcal V_i$ is then finite and its elements have diameters $<\epsilon_i$. The definition can be rewritten as follows:
\begin{multline}\label{C finite}
\forall (n_1,n_2,\dots)\in \mathbb{N}^\omega \ \exists k\in \mathbb{N} \ \exists (\mathcal V_1,\mathcal V_2,\dots,\mathcal V_k) \ \forall i\le k \\
\bigl[\text{$\mathcal V_i $ is finite, consists of open subsets of $X$ and}\\
 [V,V'\in \mathcal V_i, V\neq V' \Rightarrow V\cap V'= \emptyset]\\ \text{and}\quad [V\in \mathcal V_i,  x,y\in V \Rightarrow d(x,y)< \frac1{n_i}]
 \quad\text{and}\quad \bigcup\bigcup_{i\le k}\mathcal V_i = X\bigr].
\end{multline}

\begin{proposition}\label{p2}
For each  $n\ge 0$, the families $\mathcal C_n(Y)$ of $C$-compacta of dimensions $\ge n$ in a compact space $Y$ and $\mathcal C_n(Y)\cap C(Y)$  are coanalytic subsets of $2^Y$.
\end{proposition}
\begin{proof}
In view of~\eqref{C finite}, the definition of a $C$-compactum $(X,d)$ in terms of closed subsets of $Y$ runs as follows:
\begin{multline}\label{C comp}
\forall (n_1,n_2,\dots)\in \mathbb{N}^\omega \ \exists k\in \mathbb{N} \ \exists (\mathcal F_1,\mathcal F_2,\dots,\mathcal F_k)\in (2^{(2^Y)})^k\\
\forall i\le k \ \bigl[\text{$\mathcal F_i$ is finite and}\quad \forall F\in \mathcal F_i \ (F\subset X)\\
\text{and}\quad \forall F,F'\in \mathcal F_i \ (F\neq F' \Rightarrow F\cup F'= X)\\
\text{and}\quad \forall F\in \mathcal F_i \  (x,y\in X\setminus F \Rightarrow d(x,y)< \frac1{n_i})
\quad\text{and}\quad \bigcap\bigcup_{i\le k}\mathcal F_i = \emptyset\bigr].
\end{multline}
 Formula~\eqref{C comp} easily yields that the set $\mathcal C(Y)$ of all $C$-compacta in $Y$  is coanalytic. Hence,  each of $\mathcal C_n(Y)=\mathcal C(Y)\cap \mathcal D_n(Y)$ and $\mathcal C_n(Y)\cap C(Y)$ is also coanalytic.

\end{proof}

\begin{proposition}\label{p3}
Let $Y$ be a compact space containing a Hilbert cube. Then,
for each integer $n\ge 0$,
the families $\mathcal C_n(Y)$   and $\mathcal C_n(Y)\cap C(Y)$ are $\Pi^1_1$-complete.
\end{proposition}
\begin{proof}
The argument for the $\Pi^1_1$-hardness of each of the above families is the same as in the proof of  Proposition~\ref{co1}. So, the conclusion follows from Proposition~\ref{p2}.

\end{proof}

\begin{proposition}\label{t1}
Let $Y$ be a locally connected continuum.
\begin{enumerate}
\item
 The families
$\mathcal W_n(Y)$, $\mathcal W_n(Y)\cap C(Y)$,  $\mathcal C_n(Y)$ and  $\mathcal C_n(Y)\cap C(Y)$, $n\ge 1$, are contained in  $\mathcal D_1(Y)$,  a $\sigma Z$-set  in $2^Y$.
\item
If each non-empty open subset of $Y$ contains an $n$-cell, $n\ge 2$, then the families
$\mathcal W_n(Y)\cap C(Y)$ and $\mathcal C_n(Y)\cap C(Y)$ are contained in $\mathcal D_2(Y)\cap  C(Y)$, a  $\sigma Z$-set in $C(Y)$.
  \end{enumerate}
  \end{proposition}

\begin{proof}
 The set  $\mathcal D_1(Y)$ is a $\sigma Z$-set in $2^Y$ since there is a deformation $2^Y\times I\to 2^Y$ through finite sets~\cite{Cu}).
  The set $\mathcal D_2(Y)\cap  C(Y)$  is a $\sigma Z$-set in $C(Y)$ as it is an $F_{\sigma}$-absorber in $C(Y)$ (see Examples~\ref{ex1} (1)).

  \end{proof}

Denote by $\mathcal{SCD}_n(X)$ the family of all strongly countable-dimensional compacta of dimension $\ge n$ in a space $X$.
\begin{theorem}\label{t:Hm} Let $X$ be a locally connected continuum such that each non-empty open subset of $X$ contains a copy  of the Hilbert cube (in particular, $X$ can be a Hilbert cube manifold).

Families $\mathcal{SCD}_n(X)$, $\mathcal W_n(X)$, and $\mathcal C_n(X)$ are coanalytic absorbers in $2^X$ for $n\ge 1$.

Families  $\mathcal{SCD}_n(X)\cap C(X)$,
$\mathcal W_n(X)\cap C(X)$ and $\mathcal C_n(X)\cap C(X)$ are coanalytic absorbers in $C(X)$ for  $n\ge 2$.
\end{theorem}

\begin{proof}
We have $\mathcal{SCD}_n(X)\subset \mathcal C_n(X) \subset \mathcal W_n(X)$ (see~\cite{E}).
In view of Propositions~\ref{co1}, \ref{p2} and~\ref{t1}, it suffices to check the strong $\Pi^1_1$-universality of the families.
In the case when $X=I^\omega$
 an Approach I-construction used in the proof of~\cite[Theorem 3.1]{KS} for the strong $\Pi^1_1$-universality of $\mathcal{SCD}_n(I^\omega)$ applies to other families without any change. Its main ingredient  is a continuous mapping
$\xi:I^\omega\to C(I^\omega)$ such that
\begin{itemize}
\item
 $\xi(q)$ is a strongly countable-dimensional continuum of dimension $\ge n$ for $q\in M$ (actually, it is  a countable union of  Euclidean cubes of dimensions $\ge n$) and
 \item
  $\xi(q)$ contains a Hilbert cube for  $q\notin M$,
  \end{itemize}
 where  $M$ is a coanalytic subset of $I^\omega$. In other words, $\xi$ is a continuous reduction of $M$ to $\mathcal{SCD}_n\cap C(I^\omega)$.

Observe that $\xi$ can also be  viewed as a reduction of $M$ to any of the families in our theorem, since their members do not contain Hilbert cubes.

 In a general case, we use Approach II~\ref{ap2}. So, for each nonempty open subset $U$ of $X$, choose an embedding $\eta_U:I^\omega\to U$ and let $\varphi_U(q)=\eta_U\bigl(\xi(q)\bigr)$. We have mappings $\varphi_U:I^\omega\to C(U)$ satisfying $\varphi_U^{-1}(\mathcal A)=M$, where $\mathcal A$ is any of the considered families. Since taking  finite unions of copies $\varphi_U(q)$ and adding zero- or one-dimensional part $A(q)$ of $g(q)$ do not change dimension properties of members of each family, we have that
 $g(q)\in \mathcal A$ iff $q\in M$  for $q\in I^\omega\setminus K$.

\end{proof}

\

Other intriguing  spaces studied in dimension theory are \emph{hereditarily infinite-dimensional compacta} in the sense of Henderson, i.e., infinite-dimensional compact spaces whose nonempty closed subspaces are either infinite-dimensional or zero-dimensional.

\begin{proposition}\label{p4}
The collection $\mathcal{HID}(Y)$ of all hereditarily infinite-dimen\-sional compacta in a compact space $Y$ is a coanalytic subset of $2^Y$.
\end{proposition}

\begin{proof}
We have
$X\in\mathcal{HID}(Y)$ if and only if
\begin{equation}\label{e:HID}
\forall F\in 2^Y \ \bigl(\text{if $F\subset X$, then $\dim F=0$ or} \ \forall n \ \dim F \ge n\bigr).
\end{equation}
A direct evaluation of the projective complexity of~\eqref{e:HID} shows that $X$ is coanalytic.

\end{proof}

\begin{theorem}\label{t:HID}
The families  $\mathcal{HID}(I^\omega)$  and $\mathcal{HID}(I^\omega)\cap C(I^\omega)$ are  $\Pi^1_1$-complete subsets of $2^{I^\omega}$ and of $C(I^\omega)$, respectively.
\end{theorem}
\begin{proof}
The proof of $\Pi^1_1$-hardness will use an idea of R. Pol presented in~\cite[Lemma 6.3]{Pol}.

Take a continuum $L\in \mathcal{HID}(I^\omega)$ and the set $\mathcal L$  of all topological copies   of $L$ in $I^\omega$. Let $\mathcal P$ be the collection of all pseudoarcs in $I^\omega$. Families $\mathcal P$ and $\mathcal L$ are dense in $C(I^\omega)$ and $\mathcal P$ is a $G_\delta$-subset of $C(I^\omega)$. Let $Q$ be a countable dense subset of a Cantor set $C\subset I$.

Let $\varphi:C \to \{\mathcal P, \mathcal L\}$ be a function defined by
\[
      \varphi(x)=
\begin{cases}
  \mathcal P, & \text{for $x\in C\setminus Q$;} \\
    \mathcal L, & \text{for $x\in Q$.}
  \end{cases}
\]

 By~\cite[Corollary 1.61]{Michael},  $\varphi(x)$ has a continuous selection
$$\sigma:C \to C(I^\omega),\quad \sigma(x)\in \varphi(x).$$

Define the map $\pi:2^C\to 2^{I\times I^\omega}\stackrel{\rm top}{=} 2^{I^\omega}$ by
$$\pi(A)=\bigcup_{x\in A}\{x\}\times \sigma(x).$$
Observe that $\pi^{-1}(\mathcal{HID}(I^\omega))= 2^Q$ is the Hurewicz set of all  compacta in $Q$, a standard coanalytic complete set. In other words, $\pi$ is a continuous reduction of  $2^Q$ to $\mathcal{HID}(I^\omega)$.

If we identify to a point the Cantor set level of each $\pi(A)$ and respectively modify $\pi$, then we get a continuous reduction of  $2^Q$ to
$\mathcal{HID}(I^\omega)\cap C(I^\omega)$.

\end{proof}

\

Families $\mathcal{HID}(I^\omega)$ and  $\mathcal{HID}(I^\omega)\cap C(I^\omega)$ are contained,  respectively, in $\mathcal D_2(I^\omega)$ and $\mathcal D_2(I^\omega)\cap C(I^\omega)$  which  are $\sigma Z$-sets in $2^{I^\omega}$ and in $C(I^\omega)$.

We do not know if  $\mathcal{HID}(I^\omega)$ is strongly $\Pi^1_1$-universal in  $2^{I^\omega}$.

\bibliographystyle{amsplain}

\end{document}